\documentclass[12pt]{article}
\usepackage[british]{babel}
\usepackage{amsmath}
\usepackage{amsthm}
\usepackage{amsfonts}
\usepackage{amssymb}
\usepackage{fancyhdr}
\usepackage[all]{xy}           
\usepackage{graphicx}
\usepackage[left=3cm, right=3cm]{geometry}
\usepackage{enumerate}
\usepackage{cite}
\usepackage{hyperref} 

\theoremstyle{plain}
\newtheorem{thm}{Theorem}[section]
\newtheorem{cor}[thm]{Corollary}
\newtheorem{lem}[thm]{Lemma}
\newtheorem{pro}[thm]{Proposition}

\theoremstyle{definition}

\theoremstyle{remark}
\newtheorem{remark}[thm]{Remark}

\def\dim{\text{\rm dim}}

\author{Tsiu-Kwen Lee$^\flat$ and Jheng-Huei Lin$^\natural$}

\title{Commutators and products of Lie ideals of prime rings}
\date{}

\begin{document}

\maketitle

\centerline {Department of Mathematics, National Taiwan
University${^\flat}$}

\centerline {Taipei, Taiwan}

\centerline {tklee@math.ntu.edu.tw$^\flat$}

\centerline {and}

\centerline {Department of Mathematics, National Changhua
University of Education${^\natural}$}

\centerline {Changhua, Taiwan}

\centerline {r01221012@ntu.edu.tw${^\natural}$}

\begin{abstract}\vskip6pt
\noindent Motivated by some recent results on Lie ideals, it is proved that if $L$ is a Lie ideal of a simple ring $R$ with center $Z(R)$, then $L\subseteq Z(R)$, $L=Z(R)a+Z(R)$ for some noncentral $a\in L$, or $[R, R]\subseteq L$, which gives a generalization of a classical theorem due to Herstein. We also study commutators and products of noncentral Lie ideals of prime rings.
Precisely, let $R$ be a prime ring with extended centroid $C$. We completely characterize Lie ideals $L$ and elements $a$ of $R$ such that $L+aL$ contains a nonzero ideal of $R$. Given noncentral Lie ideals $K, L$ of $R$, it is proved that $[K, L]=0$ if and only if $KC=LC=Ca+C$ for any noncentral element $a\in L$. As a consequence, we characterize noncentral Lie ideals $K_1,\ldots,K_m$ with $m\geq 2$ such that
 $K_1K_2\cdots K_m$ contains a nonzero ideal of $R$.
Finally, we characterize noncentral Lie ideals $K_j$'s and $L_k$'s satisfying $\big[K_1K_2\cdots K_m, L_1L_2\cdots L_n\big]=0$ from the viewpoint of centralizers.
 \end{abstract}

{ \hfill\break \noindent 2020 {\it Mathematics Subject Classification.}\ 16N60, 16W10. \vskip2pt


 \vskip2pt


\section{Introduction}
The concept of Lie algebras, or infinitesimal groups, originated out of the study to  infinitesimal transformations by Lie in the 1870s.
The name ``Lie algebra" was introduced by Weyl in 1934.
See \cite[Section 49.5]{kline1990} and other historical documents.
Since then, the topic has been studied by a lot of researchers and applied to many areas of science and mathematics such as theoretical physics, geometry and operator theory.
One of the most important substructures is its ideals, i.e., Lie ideals.
In 1950, Jacobson and Rickart studied Jordan homomorphisms through the viewpoint of Lie ideals (see \cite[Sections 6 and 7]{jacobson1950}).
In order to apply their theorems to primitive rings with nonzero socle, they began to investigate the ideal structure of Lie rings and discovered important properties of noncentral Lie ideals on general matrix rings.
See \cite[Theorems 19 and 20]{jacobson1950}.
On the other hand, Herstein also began studying the Jordan ring and the Lie ring of an associative ring, including the relationship among their ideal structures (see \cite{herstein1954,herstein1955a,herstein1955b,herstein1961,herstein1969,herstein1970}).
These developed several branches of research concerning Lie and Jordan algebras, which have been used in a lot of areas of mathematics, especially the study of operator algebras (see, for example, \cite{bresar2008,bresar2020,gardella2024a,gardella2024,hopenwasser2004,marcoux2010,robert2016}), a subject having direct applications to differential geometry, representation theory, quantum mechanics, etc.
See \cite{kline1990} or related materials.
However, almost all of classical theorems require rings to be $2$-torsion free, limiting the scope of usage.
Accordingly, it is worthwhile to resume these studies with more general cases by adopting modern mathematical tools, making the range of application wider.

Throughout the paper, $R$ is an associative ring, not necessarily with unity, with center $Z(R)$. For $a, b\in R$, we let $[a, b]:=ab-ba$, the additive commutator of $a$ and $b$. Given two subsets $A, B$ of $R$, let $[A, B]$ (resp. $AB$) denote the additive subgroup of $R$ generated by all elements $[a, b]$ (resp. $ab$) for $a\in A$ and $b\in B$. We write $\overline A$ to stand for the subring of $R$ generated by $A$. An additive subgroup $L$ of $R$ is called a {\it Lie ideal} if $[L, R]\subseteq L$. A  Lie ideal $L$ is called {\it abelian} (resp. {\it central}) if $[L, L]=0$  (resp. $L\subseteq Z(R)$). Clearly, if $K$ and $L$ are Lie ideals of $R$, then so are $[K, L]$ and $KL$.

A ring $R$ is called {\it simple} if $R^2\ne 0$ and the only ideals of $R$ are $\{0\}$ and itself.  A famous theorem due to Herstein characterizes Lie ideals of simple rings as follows (see \cite[Theorem 1.5]{herstein1969}).

\begin{thm} (Herstein)\label{thm5}
Let $R$ be a simple ring. Then, given a Lie ideal $L$ of $R$, either $[R, R]\subseteq L$ or $L\subseteq Z(R)$ unless $\text{\rm char}\,R=2$ and $\dim_{Z(R)}R=4$.
 \end{thm}

 A ring $R$ is called {\it prime} if, given $a, b\in R$,
 $aRb=0$ implies that either $a=0$ or $b=0$. Given a  prime ring $R$, let $Q_s(R)$ be the Martindale symmetric ring of quotients of $R$. It is known that $Q_s(R)$ is also a prime ring, and its center, denoted by $C$, is a field, which is called the {\it extended centroid} of $R$. The extended centroid $C$ will play an important role in our present study. See \cite{beidar1996} for details.\vskip6pt

\noindent{\bf Definition.}\ A prime ring $R$ is called {\it exceptional} if both $\text{\rm char}\,R= 2$ and $\dim_CRC=4$.
\vskip6pt

Lanski and Montgomery characterized  Lie ideals $K, L$ of a prime ring $R$ satisfying $[K, L]\subseteq Z(R)$ (see \cite[Lemma 7]{lanski1972}).

\begin{lem} (Lanski-Montgomery 1972)
\label{lem5}
Let $R$ be a prime ring with Lie ideals $K, L$. If $[K, L]\subseteq Z(R)$, then one of $K$  and $L$ is central except when $R$ is exceptional.
\end{lem}

Motivated by Theorem \ref{thm5} and Lemma \ref{lem5}, the structure of Lie ideals of exceptional prime rings is studied in Section 4.
As a consequence we extend Theorem \ref{thm5} as follows (i.e., Theorem \ref{thm20}). \vskip6pt

\noindent{\bf Theorem A.}\ {\it Let $R$ be a simple ring, and let $L$ be a Lie ideal of $R$. Then
$L\subseteq Z(R)$, $L=Z(R)a+Z(R)$ for some $a\in L\setminus Z(R)$, or $[R, R]\subseteq L$.}\vskip6pt

 Given a ring $R$, we use $\widetilde{R}$ to denote its minimal unitization, given by $\widetilde{R}=R$ if $R$ is unital, and by $\widetilde{R} = R\times \mathbb{Z}$ with coordinate-wise addition and multiplication $(x, m)(y, n)=(xy+nx+my, mn)$ if $R$ is non-unital. Therefore, if $A$ is a subset  of $R$, then $\widetilde{R}A\widetilde{R}$ is equal to the ideal of $R$ generated by $A$.

The aim of the paper focuses on the study of the connection between ideals and Lie ideals.
The following is the most known result (see, for instance, \cite[Lemma 2.1]{lee2022}).

\begin{thm}\label{thm18}
Let $R$ be a ring with a Lie ideal $L$. Then ${\widetilde R}[L, L]{\widetilde R}\subseteq L+L^2$.
 \end{thm}

In Theorem \ref{thm18}, if $L$ is not abelian then $L+L^2$ contains a nonzero ideal of $R$.  We will extend the result to the context of prime rings.
Precisely, we completely characterize Lie ideas $L$ and elements $a$ of a prime ring $R$ such that $L+aL$ contains a nonzero ideal of $R$.  In addition to being used in the sequel (see Proposition \ref{pro2}), this theorem (i.e., Theorem \ref{thm16}) is also very interesting in itself.\vskip6pt

\noindent{\bf Theorem B.}\ {\it Let $R$ be a prime ring with a noncentral Lie ideal $L$, and $a\in R$.

\noindent Case 1:\ $L$ is not abelian. If $a\in R\setminus Z(R)$, then $L+aL$ contains a nonzero ideal of $R$;

\noindent Case 2:\  $L$ is abelian. Then $LC=Cb+C$ for some $b\in L\setminus Z(R)$ and the following hold:

(i)\ If $a\in LC$, then $L+aL$ does not contain any nonzero ideal of $R$;

(ii)\ If $a\in [RC, RC]\setminus LC$, then $L+aL$ contains a nonzero ideal of $R$;

(iii)\ If $a\notin [RC, RC]$, then $L+aL$ contains a nonzero ideal of $R$ if and only if $[a, b]$ is a unit of $RC$.}
\vskip6pt

The following is a special case of \cite[Theorem 1.4]{lee2022}.

\begin{thm}\label{thm6}
Let $R$ be a ring satisfying the property that every proper ideal of $R$ is contained in a maximal ideal of $R$. If $\widetilde{R}[R, R]\widetilde{R}=R$, then $R=[R, R]+[R, R]^2$.
\end{thm}

We remark that Theorem \ref{thm6} still holds if the given property is replaced by the property that every proper ideal
is contained in a prime ideal of $R$. Moreover, we can further get $R=[R, R]^2$ in Theorem \ref{thm6} (see \cite[Theorem 3.10]{calugareanu2024}). In fact, applying the argument
given in the proof of \cite[Theorem 3.10]{calugareanu2024}, we can prove that if $L$ is a Lie ideal of $R$ satisfying $L=[L, L]+[L, L]^2$, then $[L, L]\subseteq [L, L]^2$ and so $L=[L, L]^2$.
It is natural to ask the question:\ If $K$ is a Lie ideal of $R$, does then $K^2$ contain a nonzero ideal?
The following is due to \cite[Lemma 3.2 (4)]{gardella2024-1}.

\begin{lem}\label{lem1}
Let $R$ be a ring with a Lie ideal $L$.
Then $\widetilde{R}[L,L^2]\widetilde{R} \subseteq L^2$.
\end{lem}

A ring $R$ is called {\it semiprime} if, given $a\in R$, $aRa=0$ implies $a=0$.
Given a Lie ideal $L$ in a semiprime ring $R$, it follows from \cite[Proposition 2.1]{gardella2024}
that $[L, L^2]=0$ if and only if $L$ is abelian. Together with Lemma \ref{lem1}, we have the following.

\begin{thm}\label{thm3}
Let $R$ be a semiprime ring with a Lie ideal $L$, which is not abelian. Then $L^2$ contains a nonzero ideal of $R$.
Precisely, we have $0\ne \widetilde{R}[L, L^2]\widetilde{R} \subseteq L^2$.
\end{thm}

Our next aim is to study commutators and products of noncentral Lie ideals of prime rings.
Motivated by Lemma \ref{lem5} and Theorem \ref{thm3}, it is natural to study noncentral Lie ideals $K, L$ of a semiprime ring $R$ satisfying the identity $[K^m, L^n]=0$ for some positive integers $m, n$. Our study is also reduced to the prime case since every semiprime ring is a subdirect product of prime rings.
The following is Theorem \ref{thm2}. \vskip6pt

\noindent{\bf Theorem C.} {\it Let $R$ be a prime ring with noncentral Lie ideals $K, L$, and positive integers $m, n$. Then the following are equivalent:

(i)\ $[K^m, L^n]=0$;

(ii)\ $[K, L]=0$;

 (iii)\ $R$ is exceptional and $KC=LC=Ca+C$ for any noncentral element $a\in L$.}\vskip6pt

Motivated by Theorem \ref{thm3}, given noncentral Lie ideals $K$ and $L$ in a prime ring $R$, we consider whether the product $K^mL^n$ contains a nonzero ideal of $R$, where $m, n$ are positive integers.
A complete solution is given as follows.\vskip6pt

\begin{thm}\label{thm12}
Let $R$ be a prime ring with noncentral Lie ideals $K, L$, and positive integers $m, n$.
Then $K^mL^n$ contains a nonzero ideal of $R$ except when $KC=LC=Ca+C$ for any noncentral element $a\in K$.
\end{thm}

Actually, we will prove the following more general theorem. This is Theorem \ref{thm11}. We also note that Theorem \ref{thm12} is its immediate consequence.\vskip6pt

\noindent{\bf Theorem D.} {\it Let $R$ be a prime ring with noncentral Lie ideals $K_1,\ldots,K_m$ with $m\geq 2$.
Then $K_1K_2\cdots K_m$ contains a nonzero ideal of $R$ except when, for $1\leq j\leq m$, $K_1C=K_jC=Ca+C$ for any noncentral element $a\in K_1$.}
\vskip6pt

Finally, we extend Theorem C to a more general form  from the viewpoint of centralizers. This is Theorem \ref{thm13}.\vskip6pt

\noindent{\bf Theorem E.} {\it Let $R$ be a prime ring with noncentral Lie ideals $K_1,\ldots,K_m, L_1, \ldots,L_n$, where $m, n\geq 1$.
Then $\big[K_1\cdots K_m, L_1\cdots L_n\big]=0$ if and only if, for all $j, k$, we have $K_1C=K_jC=L_kC=Ca+C$ for any noncentral element $a\in K_1$.}
\vskip6pt

Whenever it is more convenient, we will use the widely accepted shorthand “iff”
for “if and only if” in the text.

\section{Preliminaries}
 We begin with some basic properties of rings.

\begin{lem}\label{lem8} Let $R$ be a ring.

(i)\ If $A$ is an additive subgroup of $R$, then $[A, R]=[\overline A, R]$.\vskip6pt

(ii)\ If $R$ is semiprime and if $[a, [R, R]]=0$ where $a\in R$, then $a\in Z(R)$.

(iii)\ If $R$ is prime and if $a[b, R]\subseteq Z(R)$ where $a, b\in R$, then $a=0$ or $b\in Z(R)$.

(iv)\ $\widetilde{R}L\widetilde{R}=\widetilde{R}L$ for any Lie ideal $L$ of $R$.

(v)\ If $L$ is a Lie ideal of $R$ and $m>1$ is a positive integer, then
$$
\widetilde{R}[L^{m-1}, L^m]\widetilde{R}\subseteq L^m\ \text{\rm and}\ \ \big[\widetilde{R}[L^{m-1}, L^m]\widetilde{R}, R]\subseteq L^m\cap L.
$$

(vi)\  If $L$ is a Lie ideal of $R$, then $\big[R, {\widetilde R}[L, L]{\widetilde R}\big]\subseteq L$.
\end{lem}

\begin{proof}
(i)\ Let $a_1,\ldots,a_n\in A$ and $x\in R$. By induction on $n$, we get
\begin{equation*}
[a_1a_2\cdots a_n, x]=[a_2\cdots a_n, xa_1]+ [a_1, a_2\cdots a_nx]\in [A, R].
\end{equation*}
Thus $[A, R]=[\overline A, R]$.

(ii)\ See \cite[Lemma 1.5]{herstein1969}.

(iii)\ Let $x\in R$. Then
$a[b, xb]=a[b, x]b\in Z(R)$.
Since $a[b, x]\in Z(R)$, either $a[b, x]=0$ or $b\in Z(R)$. In either case, we have $a[b, x]=0$.
Therefore, $a[b, x]=0$ for all $x\in R$. Then $a[b, Rx]=0$ for all $x\in R$ and so $aR[b, x]=0$ for all $x\in R$.
The primeness of $R$ implies that   either $a=0$ or $b\in Z(R)$.

(iv)\ Note that $xay=x[a, y]+xya$ for $x, y\in \widetilde{R}$ and $a\in L$. Therefore,
$$
\widetilde{R}L\widetilde{R}\subseteq \widetilde{R}[L, \widetilde{R}]+\widetilde{R}^2L\subseteq \widetilde{R}L,
$$
as desired.

(v)\ Note that $[L^{m-1}, L^m]$ is a Lie ideal of $R$.
It follows from (i) that
 \begin{eqnarray}
[\widetilde{R}, L^m]=[R, L^m]\subseteq [R, \overline L]=[R, L]\subseteq L.
\label{eq:6}
\end{eqnarray}
By (iv) and the fact that $ [\widetilde{R}L^{m-1}, L^m]\subseteq [R, L^m]\subseteq L^m$, we have
 \begin{eqnarray*}
\widetilde{R}[L^{m-1}, L^m]\widetilde{R} &=&\widetilde{R}[L^{m-1}, L^m]\\
&\subseteq & [\widetilde{R}L^{m-1}, L^m]+[\widetilde{R}, L^m]L^{m-1}\\
&\subseteq &L^m+LL^{m-1}\\
&= &L^m,
 \end{eqnarray*}
where the first inclusion above follows from the fact that $x[a, b]=[xa, b]-[x, b]a$ for $x, a, b\in R$.
Finally,
$\big[\widetilde{R}[L^{m-1}, L^m]\widetilde{R}, R\big]\subseteq [L^m, R]\subseteq L$ by Eq.\eqref{eq:6}.

(vi)\ In view of Theorem \ref{thm18}, we have ${\widetilde R}[L, L]{\widetilde R}\subseteq L+L^2$. By (i) we have
$$
\big[R, {\widetilde R}[L, L]{\widetilde R}\big]\subseteq [R, L+L^2]\subseteq [R, \overline L]=[R, L]\subseteq L,
$$
as desired.
\end{proof}

\begin{lem} (\cite[Lemma 4.3]{calugareanu2024})
Let $R$ be a prime ring. If $a\in
RC\setminus C$, then $\dim_C\big[a, RC\big]>1$.
\label{lem14}
\end{lem}

\begin{lem}
\label{lem4}
Let $R$ be a ring with a Lie ideal $L$. If $L+V=R$, where $V$ is an additive subgroup of $R$ satisfying $[V, V]=0$, then $[R, R]\subseteq L$.
\end{lem}

\begin{proof}
Since $[V, V]=0$, we have
$$
[R, R]=[L+V, L+V]\subseteq [L, L]+[L, V]\subseteq L,
$$
as desired.
\end{proof}

Let $R$ be a prime PI-ring. It is known that $Z(R)\ne 0$ and $C$ is equal to the quotient field of $Z(R)$.
In this case, $RC$ is a finite-dimensional central simple algebra. See \cite[Corollary 1]{rowen1973} for details. Therefore, the following lemma is reduced to the case that $R$ is a simple ring (see \cite[Lemma 2.3]{lee2022}).

\begin{lem} \label{lem7}
Let $R$ be a prime ring. Then
$0\ne \big[[R, R], [R, R]\big]\subseteq Z(R)$ iff $R$ is exceptional.
\end{lem}

\begin{lem} \label{lem11}
Let $R$ be a exceptional prime ring, $a, b\in R$. If $0\ne [a, b]\in Z(R)$, then $a, b\in [RC, RC]$.
\end{lem}

\begin{proof}
Note that $RC$ is a $4$-dimensional central simple algebra. Given $x\in RC$, $x^2=\alpha x+\beta$ for some $\alpha, \beta\in C$.
It is known that $x\in [RC, RC]$ iff $\alpha=0$. Suppose on the contrary that $a\notin [RC, RC]$. Then $a^2=\mu_1a+\mu_2$ for some $\mu_1, \mu_2\in C$ with $\mu_1\ne 0$.
Then
$$
0=[a, [a, b]]=[a^2, b]=[\mu_1a+\mu_2, b]=\mu_1[a, b]
$$
and so $[a, b]=0$, a contradiction. Similarly, we have $b\in  [RC, RC]$.
\end{proof}

\section{Semiprime rings}
Throughout this section, {\it $R$ is always a semiprime ring with Lie ideals $K, L$}. We begin with a key lemma.\vskip6pt

\begin{lem}\label{lem6}
Let $a\in R$ and $m$ a positive integer.
If $[a, L^m]=0$, then $[a, L]=0$.
\end{lem}

\begin{proof}
Since every semiprime ring is a subdirect product of prime rings, we may assume that $R$ is a prime ring.
Clearly, we may assume that $m>1$ and proceed the proof by induction on $m$. Since $L^m$ is a Lie ideal of $R$, we get
$[L^m, R]\subseteq L^m$ and $[L^m, R]\subseteq L$ by Eq.\eqref{eq:6}.
Therefore,
$$
\big[a, L^{m-1}[L^m, R]\big]\subseteq [a, L^m]=0\ \text{\rm and}\ \big[a, [L^m, R]\big]\subseteq [a, L^m]=0.
$$
Hence $[a, L^{m-1}][L^m, R]=0$. By  Lemma \ref{lem8} (iii), either $[a, L^{m-1}]=0$ or $L^m\subseteq Z(R)$.

We first consider the latter case. Then $L^{m-1}[L, R]\subseteq L^m\subseteq Z(R)$.
By Lemma \ref{lem8} (iii) again, either $L^{m-1}=0$  or $[L, R]=0$.
In either case, we always have $[a, L^{m-1}]=0$. By induction on $m$, we conclude that $[a, L]=0$, as desired.
\end{proof}

\begin{cor}\label{cor1}
If $L^m\subseteq Z(R)$ where $m$ is a positive integer, then $L\subseteq Z(R)$.
\end{cor}

\begin{cor}\label{cor2}
If $[K^m, L^n]=0$ for some positive integers $m, n$, then $[K, L]=0$.
\end{cor}

\begin{proof}
Let $a\in K^m$. By assumption, $[a, L^m]=0$ and so $[a, L]=0$ (see Lemma \ref{lem6}). That is, $[K^m, L]=0$.
Applying the same argument, we get $[K, L]=0$.
\end{proof}

We refer the reader to \cite[Proposition 2.1]{gardella2024} for $K=L$, $m=1$ and $n=2$.

\begin{cor}\label{cor3}
If $\big[[R, R]^m, [R, R]^n\big]=0$ where $m, n$ are two positive integers, then $R$
is commutative.
\end{cor}

\begin{proof}
In view of Corollary \ref{cor2}, $\big[[R, R], [R, R]\big]=0$ and so $R$ is commutative (see Lemma \ref{lem8} (ii)).
\end{proof}

See also \cite[Theorem 2.3]{gardella2024} for the case $m=1$ and $n=2$.
The next theorem extends Theorem \ref{thm3} to a more general form.

\begin{thm}\label{thm8}
Suppose that $L$ is not abelian. Then $L^m$ contains a nonzero ideal of $R$ for any positive integer $m>1$.
Precisely, we have $0\ne \widetilde{R}[L^{m-1}, L^m]\widetilde{R} \subseteq L^m$.
\end{thm}

\begin{proof}
Since $L$ is not abelian, the semiprimeness of $R$ implies that $[L^{m-1}, L^m]\ne0$ (see Corollary \ref{cor2}) and hence by Lemma \ref{lem8} (v),
$0\ne \widetilde{R}[L^{m-1}, L^m]\widetilde{R} \subseteq L^m$, as desired.
\end{proof}

\section{Exceptional prime rings}

In this section we focus on the Lie ideal structure of exceptional prime rings (see Theorem \ref{thm19})
and then give the proof of Theorem A (i.e., Theorem \ref{thm20}).
We begin with characterizing noncentral abelian Lie ideals of prime rings.

\begin{lem}\label{lem131}
Let $R$ be a prime ring with a noncentral Lie ideal $L$. Then the following are equivalent:

(i)\ $L$ is abelian;

(ii)\ $R$ is exceptional and $LC=[a, RC]=Ca+C$ for any $a\in L\setminus Z(R)$;

(iii)\ $\dim_CLC=2$.
\end{lem}

\begin{proof}
The implication (ii) $\Rightarrow$ (i) is clear.
We now prove that (i) implies (ii).
Since $L$ is abelian, it follows from Lemma \ref{lem5} that $R$ is exceptional.
Let $a\in L\setminus Z(R)$. Then
$[a, RC]\subseteq LC$.
By Lemma \ref{lem14}, it follows that $\dim_C[a, RC]>1$.
We claim that $LC= [a, RC]$. Otherwise, $\dim_CLC\geq 3$. In this case, $RC=LC+Cb$ for some $b\in R$.
In view of Lemma \ref{lem4}, we have
$[RC, RC]\subseteq LC$.
Hence
$$
\big[[RC, RC], [RC, RC]\big]\subseteq [LC, LC]=0.
$$
By Lemma \ref{lem8} (ii), $[RC, RC]\subseteq C$ and so $\big[[RC, RC], RC\big]=0$. So $RC$ is commutative, a contradiction.
Therefore, $LC=[a, RC]$ and $\dim_CLC=\dim_C[a, RC]=2$.

We claim that $LC=[a, RC]=Ca+C$. Indeed, since $a\notin C$ and $a\in LC=[a, RC]$, it follows from Lemma \ref{lem8} (ii) and Lemma \ref{lem7} that
$$
0\ne [a, [RC, RC]]\subseteq \big[[RC, RC], [RC, RC]\big]\subseteq C.
$$
 Thus $[a, r]=1$ for some $r\in RC$.
Given $\beta, \gamma\in C$, we have
$$
\beta a+\gamma =[a, (\beta a+\gamma)r]\in [a, RC]=LC,
$$
implying that $LC=[a, RC]=Ca+C$ since $\dim_CLC=\dim_C[a, RC]=2$, as claimed. Hence (ii) is proved.

Clearly, (ii) implies (iii). Finally, we prove that (iii) implies (i).
Assume that (iii) holds, that is, $\dim_CLC=2$.
Choose a noncentral element $a\in L$. By Lemma \ref{lem14}, $LC=[a, RC]$. Thus
$[a, r]=a$
for some $r\in RC$. Write $LC=Ca+Cb$ for some $b\in L$. Then $[LC, LC]=C[a, b]$ and so $\big[[L, L], [L, L]\big]=0$. In view of Lemma \ref{lem5},  $R$ is exceptional.

It follows from Lemma \ref{lem7} that
$$
0\ne [a, [RC, RC]]\subseteq LC\cap \big[[RC, RC], [RC, RC]\big]\subseteq C.
$$
 Hence $Ca+C\subseteq [a, RC]=LC$. This implies that
$LC=Ca+C$ and so $L$ is abelian, as desired.
\end{proof}

\begin{cor}\label{cor4}
Let $R$ be a prime ring with a noncentral Lie ideal $L$, and $k>1$ a positive integer.
Then $L^k$ contains a nonzero ideal of $R$ except when $LC=Ca+C$ for any $a\in L\setminus Z(R)$.
\end{cor}

\begin{proof}
In view of Lemma \ref{lem131}, $LC=Ca+C$ for any $a\in L\setminus Z(R)$ iff $L$ is abelian.
The corollary follows directly from Theorem \ref{thm8}.
\end{proof}

\begin{lem}\label{lem135}
If $a$ is not central in a prime ring $R$, then $Ca+C$ is a Lie ideal of $RC$ iff $[a, RC]=Ca+C$.
\end{lem}

\begin{proof}
``$\Rightarrow$":\ \ Since $Ca+C$ is a Lie ideal of $RC$, we get
$$
[a, RC]=[Ca+C, RC]\subseteq Ca+C.
$$
By Lemma \ref{lem14}, $[a, RC]=Ca+C$.
The converse implication is clear.
\end{proof}

The following gives another characterization of exceptional prime rings.

\begin{thm}\label{thm139}
Let $R$ be a noncommutative prime ring. Then the following are equivalent:

(i)\ $Ca+C$ is a Lie ideal of $RC$  for any $a\in [RC, RC]\setminus C$;

(ii)\ $R$ is exceptional.
\end{thm}

\begin{proof}
Assume that (i) holds.
Since $R$ is a noncommutative prime ring, we have $[RC, RC]\setminus C\ne \emptyset$ (see Lemma \ref{lem8} (ii)). Choose $a\in [RC, RC]\setminus C$.
Thus $Ca+C$ is a Lie ideal of $RC$, which is of dimension $2$. In view of Lemma \ref{lem131}, $R$ is exceptional.

Conversely, assume that (ii) holds, that is, $R$ is exceptional.
Let  $a\in [RC, RC]\setminus C$.
Then
$$
0\ne \big[a, [RC, RC]\big]\subseteq \big[[RC, RC], [RC, RC]\big]\subseteq C,
$$
where the second inclusion follows from Lemma \ref{lem7}.
Choose $r\in [RC, RC]$ such that $[a, r]=1$.
Let $\beta, \gamma\in C$. Then $\beta a+\gamma=[a, (\beta a+\gamma)r]\in [a, RC]$.
This implies that $Ca+C\subseteq [a, RC]$.
On the other hand,
$ [RC, RC]+Cb=RC$
for some $b\in RC$. Then
$$
[a, RC]=[a,  [RC, RC]+Cb]=C+C[a, b].
$$
Therefore, $\dim_C[a, RC]=2$ and so $[a, RC]=Ca+C$. In view of Lemma \ref{lem135}, $Ca+C$ is a Lie ideal of $RC$, as desired.
\end{proof}

The following characterizes Lie ideals of exceptional prime rings.

\begin{thm}\label{thm19}
Let $R$ be an exceptional prime ring, and let $L$ be a Lie ideal of $R$. Then
$L\subseteq Z(R)$, $LC=[a, RC]=Ca+C$ for some $a\in L\setminus Z(R)$, or $[RC, RC]\subseteq LC$.
\end{thm}

\begin{proof}
Suppose that $L$ is noncentral. If $\dim_CLC=2$, by Lemma \ref{lem131} we see that $LC=[a, RC]=Ca+C$  for some $a\in L\setminus Z(R)$.
Suppose that $\dim_CLC>2$. Then $RC=LC+Cb$ for some $b\in RC$. In view of Lemma \ref{lem4}, $[RC, RC]\subseteq LC$, as desired.
\end{proof}

Note that $RC=R$ if $R$ is a simple ring. Together with Herstein's theorem on simple rings (see Theorem \ref{thm5}), the following characterizes Lie ideals of simple rings (i.e., Theorem A).

\begin{thm}\label{thm20}
Let $R$ be a simple ring, and let $L$ be a Lie ideal of $R$. Then
$L\subseteq Z(R)$, $L=Z(R)a+Z(R)$ for some $a\in L\setminus Z(R)$, or $[R, R]\subseteq L$.
\end{thm}

\begin{proof}
Case 1:\ $R$ is not exceptional. By Herstein's theorem (see Theorem \ref{thm5}), either $L\subseteq Z(R)$ or $[R, R]\subseteq L$.

Case 2:\ $R$ is exceptional. In this case, we have $C=Z(R)$ and $RC=R$. In view of Theorem \ref{thm19}, $L\subseteq Z(R)$, $LC=[a, RC]=Ca+C$ for some $a\in L\setminus Z(R)$, or $[RC, RC]\subseteq LC$.

For the second case, $LC=Ca+C=[a, RC]=[a, R]\subseteq L$ and so $L=Ca+C=Z(R)a+Z(R)$.
We consider the last case that $[RC, RC]\subseteq LC$. Since $RC=R$, we get $[R, R]\subseteq LC$. Then
 \begin{eqnarray}
\big[R, [R, R]\big]\subseteq [R, LC]=[RC, L]=[R, L]\subseteq L.
\label{eq:10}
\end{eqnarray}
Since $\big[[R, R], [R, R]\big]\ne 0$ (see Lemma \ref{lem8} (ii)), it follows from the  simplicity of $R$ that $R=R\big[[R, R], [R, R]\big]R$.
By Theorem \ref{thm18}, Lemma \ref{lem8} (i), and Eq.\eqref{eq:10}, we get
$$
[R, R]=\Big[R, R\big[[R, R], [R, R]\big]R\Big]\subseteq \big[R, \overline{[R, R]}\big]=\big[R, [R, R]\big]\subseteq L,
$$
as desired.
\end{proof}

\section{Theorem B}
Let $R$ be a ring. Given a subset $A$ of $R$, we denote by $\mathfrak{C}_R(A)$ the centralizer of $A$ in $R$, that is,
$$
\mathfrak{C}_R(A)=\{x\in R\mid ax=xa\ \forall a\in A\}.
$$
If $A=\{a\}$ for some $a\in R$, we will write $\mathfrak{C}_R(a)$ instead of $\mathfrak{C}_R(\{a\})$.

\begin{lem}\label{lem9}
Let $R$ be an exceptional prime ring. If $a\in R\setminus Z(R)$, then $\mathfrak{C}_{RC}(a)=Ca+C$.
\end{lem}

\begin{proof}
Consider the $C$-linear map $\phi\colon RC\to [a, RC]$, which is defined by $\phi(x)=[a, x]$ for $x\in RC$.
Clearly, $\mathfrak{C}_{RC}(a)$ is the kernel of $\phi$. Thus
$$
RC/\mathfrak{C}_{RC}(a)\cong [a, RC].
$$
Clearly, $Ca+C\subseteq \mathfrak{C}_{RC}(a)$ and, by Lemma \ref{lem14}, $\dim_C[a, RC]>1$.
Since $\dim_CRC=4$, we have $\mathfrak{C}_{RC}(a)=Ca+C$, as desired.
\end{proof}

We are now ready to prove the following (i.e., Theorem B).

\begin{thm}\label{thm16}
Let $R$ be a prime ring with a noncentral Lie ideal $L$, and $a\in R$.

\noindent Case 1:\ $L$ is not abelian. If $a\in R\setminus Z(R)$, then $L+aL$ contains a nonzero ideal of $R$;

\noindent Case 2:\  $L$ is abelian. Then $LC=Cb+C$ for some $b\in L\setminus Z(R)$ and the following hold:

(i)\ If $a\in LC$, then $L+aL$ does not contain any nonzero ideal of $R$;

(ii)\ If $a\in [RC, RC]\setminus LC$, then $L+aL$ contains a nonzero ideal of $R$;

(iii)\ If $a\notin [RC, RC]$, then $L+aL$ contains a nonzero ideal of $R$ iff $[a, b]$ is a unit of $RC$.
\end{thm}

\begin{proof}
{Case 1}:\  Assume that $L$ is not abelian and $a\notin Z(R)$. In particular, $L$ is a noncentral Lie ideal of $R$. We claim that $[a, L]\ne 0$. Otherwise, we have $[a, L]=0$.
In view of \cite[Lemma 8]{lanski1972}, $R$ is exceptional. Suppose that $\dim_CLC\geq 3$. Then $[RC, RC]\subseteq LC$ (see the proof of Lemma \ref{lem131}) and so $\big[a, [R, R]\big]=0$. In view of Lemma \ref{lem8} (ii), it follows that $a\in Z(R)$, a contradiction. Hence $\dim_CLC=2$. By
Lemma \ref{lem131}, $L$ is abelian, a contradiction. Up to now, we have proved that $[a, L]\ne 0$.

Note that $[a, \ell]r=[ar, \ell]-a[r, \ell]$ for $r\in R$ and $\ell\in L$. Therefore,
$$
0\ne [a, L]R\subseteq [aR, L]+a[R, L]\subseteq L+aL.
$$
Since $\big[L^2, [a, L]R\big]\subseteq [\overline L, R]=[L, R]\subseteq L$ (see Lemma \ref{lem8} (i)), we see that
 \begin{eqnarray}
L^2[a, L]R\subseteq \big[L^2, [a, L]R\big]+[a, L]RL^2\subseteq L+aL.
\label{eq:11}
\end{eqnarray}
In view of Theorem \ref{thm3}, $L^2$ contains a nonzero ideal, say $I$, of $R$. Hence
$L+aL$ contains the nonzero ideal $I[a, L]R$ of $R$, as desired.

{Case 2}:\  $L$ is abelian. It follows from Lemma \ref{lem131} that
$R$ is exceptional and
$$
LC=[b, RC]=Cb+C
$$
for some $b\in L\setminus Z(R)$.

(i)\ Assume that $a\in LC$. Then $(L+aL)C=LC$. Since $\dim_CLC=2$, we get $LC\ne RC$ and so $L+aL$ does not contain a nonzero ideal of $R$.

(ii) Assume that $a\in [RC, RC]\setminus LC$.
Then, by Lemma \ref{lem7}, we have
$$
[a, b]\in \big[[RC, RC], [RC, RC]\big]\subseteq C.
$$
Since $a\notin LC$, this implies that $[a, b]\ne 0$ (see Lemma \ref{lem9}). Therefore,
$$
0\ne [a, b]R\subseteq [b, aR]+a[b, R]\subseteq [b, R]+a[b, R]\subseteq L+aL.
$$
That is, $L+aL$ contains the nonzero ideal $[a, b]R$ of $R$, as desired.

(iii) Assume that $a\notin [RC, RC]$. It follows from Lemma \ref{lem9} that $[a, b]\ne 0$. In view of Lemma \ref{lem11}, we have $[a, b]\notin C$. Clearly, $[a, b]\in L$. This implies that
$$
LC=[[a, b], RC]=C[a, b]+C.
$$
We claim that
 \begin{eqnarray}
[a, b]R+R[a, b]\subseteq L+aL\subseteq [a, b]RC+RC[a, b].
\label{eq:7}
\end{eqnarray}
From the proof of (ii), it follows that $[a, b]R\subseteq L+aL$. Since $L+aL=L+La$, we get $R[a, b]\subseteq L+La=L+aL$.
Therefore, $[a, b]R+R[a, b]\subseteq L+aL$. In order to prove the second inclusion of Eq.\eqref{eq:7}, we compute
 \begin{eqnarray}
L+aL&\subseteq&LC+aLC\nonumber\\
         &=&\big[[a, b], RC\big]+a\big[[a, b], RC\big]\nonumber\\
         &\subseteq& [a, b]RC+RC[a, b]+a[a, b]RC+aRC[a, b]\nonumber\\
         &\subseteq& [a, b]RC+RC[a, b]+a[a, b]RC.
\label{eq:9}
\end{eqnarray}
Note that $a^2=\mu_1a+\mu_2$ for some $\mu_1, \mu_2\in C$. Then $[a, [a, b]]=[a^2, b]=\mu_1[a, b]$,
and so $a[a, b]=[a, b](a+\mu_1)$.
Therefore,
$$
a[a, b]RC=\big([a, b](a+\mu_1)\big)RC\subseteq [a, b]RC.
$$
By Eq.\eqref{eq:9}, we see that $L+aL\subseteq [a, b]RC+RC[a, b]$ and so Eq.\eqref{eq:7} is proved.

We first consider the case that $L+aL$ contains a nonzero ideal of $R$. Then, by Eq.\eqref{eq:7}, we have
$$
RC=(L+aL)C=[a, b]RC+RC[a, b].
$$
Since $RC$ is a simple Artinian ring, every one-sided ideal of $RC$ is generated by one idempotent. Write $[a, b]RC=eRC$ and $RC[a, b]=RCf$
for some idempotents $e, f\in RC$. Then
 \begin{eqnarray*}
(1-e)RC(1-f)&=&(1-e)([a, b]RC+RC[a, b])(1-f)\\
                      &=&(1-e)(eRC+RCf)(1-f)\\
                      &=&0.
\end{eqnarray*}
Thus, either $e=1$ or $f=1$. In either case, $[a, b]$ is a unit of $RC$.

Conversely, assume that $[a, b]$ is a unit of $RC$. In this case, it is clear that as a ring, $[a, b]R$ is a prime ring. Since $RC$ is a PI-ring, so is the prime ring $[a, b]R$. In view of \cite[Corollary 1]{rowen1973}, $Z([a, b]R)\ne 0$. Choose $0\ne \beta\in Z([a, b]R)$. Let $x\in R$. Then $[a, b]RCx\subseteq  [a, b]RC$ and so
$[\beta, [a, b]RCx]=0$. Since $[a, b]$ is a unit of $RC$, we get $[a, b]RC=RC$. Therefore, $[\beta, RCx]=0$ and so $[\beta, x]=0$ as $1\in RC$. That is, $\beta\in Z(R)$. Hence $\beta R$ is a nonzero ideal of $R$ and $\beta R\subseteq [a, b]R\subseteq L+aL$, as desired.
\end{proof}

\begin{cor}\label{cor8}
Let $R$ be a semiprime ring, and let $L$ be a Lie ideal of $R$ such that $[a, L]\ne 0$ for some $a\in L$.
Then $L+aL$ contains a nonzero ideal of $R$.
\end{cor}

\begin{proof}
It follows from Eq.\eqref{eq:11} that
$L^2[a, L]R\subseteq L+aL$. Hence, by Theorem \ref{thm3}, we get $\widetilde{R}[L,L^2]\widetilde{R} \subseteq L^2$ and so
$$
N:=\widetilde{R}[L, L^2]\widetilde{R}[a, L]R\subseteq L+aL.
$$
To prove $N$ to be nonzero, it suffices to claim that $N_0:=R[a, L^2]R[a, L]R\ne 0$ as $N_0\subseteq N$. Otherwise, $N_0=0$ follows.

Let $P$ be a prime ideal of $R$, and let ${\widehat R}:=R/P$. We write $\hat x=x+P\in \widehat R$. Since $N_0=0$, we get
$$
{\widehat R}[\hat a, {\widehat L}^2]{\widehat R}[\hat a, {\widehat L}]{\widehat R}=\hat 0,
$$
where $\widehat L=L+P/P\subseteq {\widehat R}$. By the primeness of ${\widehat R}$, we get that either $[\hat a, {\widehat L}^2]=\hat 0$
or $[\hat a, {\widehat L}]=\hat 0$. By Lemma \ref{lem6}, the first case also implies that $[\hat a, {\widehat L}]=\hat 0$.

Since $P$ is an arbitrary prime ideal of $R$, the semiprimeness of $R$ implies that $[a, L]=0$, a contradiction. This completes the proof.
\end{proof}

\begin{remark}\label{remark1}
{\rm In Theorem \ref{thm16} we cannot conclude that $[L, L]+a[L, L]$ contains a nonzero ideal of $R$.
For instance, let $R$ be an exceptional prime ring, and let $L:=[R, R]$. Then $0\ne [L, L]\subseteq Z(R)$ (see Lemma \ref{lem7}). Therefore, given any $a\in R$,
$[L, L]+a[L, L]$ does not contain any nonzero ideal of $R$.}
\end{remark}

\begin{thm}\label{thm17}
Let $R$ be a prime ring with a Lie ideal $L$, and $a\in R\setminus Z(R)$.
If $[L, L]=[L, R]\ne 0$, then $[L, L]+a[L, L]$ contains a nonzero ideal of $R$.
\end{thm}

\begin{proof}
By the fact that $[L, L]=[L, R]\ne 0$, $[L, L]$ is a noncentral Lie ideal of $R$.
Since $L$ is not abelian, it follows from Theorem \ref{thm3} that $I\subseteq L^2$ for some nonzero ideal $I$ of $R$.
Therefore, by Lemma \ref{lem8} (i), we have
$$
[I, I]\subseteq [I, R]\subseteq [L^2, R]\subseteq [\overline L, R]=[L, R]=[L, L].
$$
As a nonzero ideal of a noncommutative prime ring $R$, $I$ itself is a noncommutative prime ring. By Lemma \ref{lem8} (ii),
it follows that $\big[[I, I], [I, I]\big]\ne 0$. Thus the Lie ideal $[L, L]$ is not abelian. Applying Theorem \ref{thm16}, we conclude that
$[L, L]+a[L, L]$ contains a nonzero ideal of $R$.
\end{proof}

Let $R$ be a ring. We let $E(R)$ stand for the additive subgroup of $R$ generated by all idempotents of $R$.
It is well-known that $E(R)$ is a Lie ideal of $R$ and $[E(R), E(R)]=[E(R), R]$ (see \cite[Lemma 3.12]{calugareanu2024}).
In particular, if $R$ is a prime ring containing a nontrivial idempotent, then $E(R)\nsubseteq Z(R)$ and so $[E(R), E(R)]\ne 0$.
Therefore, the following is an immediate consequence of Theorem \ref{thm17}.

\begin{cor}\label{cor6}
Let $R$ be a prime ring  containing a nontrivial idempotent, and $a\in R\setminus Z(R)$. Then $[E(R), E(R)]+a[E(R), E(R)]$ contains a nonzero ideal of $R$.
\end{cor}

\begin{cor}\label{cor5}
Let $R$ be a prime ring, $a\in R\setminus Z(R)$. Then $[R, R]+a[R, R]$ contains a nonzero ideal of $R$.
\end{cor}

\begin{proof}
Since $R$ is a noncommutative prime ring, we see that $[R, R]$ is a Lie ideal of $R$, which is not abelian.
By Theorem \ref{thm16}, $[R, R]+a[R, R]$ contains a nonzero ideal of $R$.
\end{proof}

\section{Theorem C}
We begin with a characterization of noncentral Lie ideals $K, L$ of a prime ring $R$ satisfying $[K, L]=0$.

\begin{lem}\label{lem133}
Let $R$ be a prime ring with noncentral Lie ideals $K, L$. If $[K, L]=0$ then
$KC=LC=Ca+C$ for any $a\in L\setminus Z(R)$.
\end{lem}

\begin{proof}
Since $K$ and $L$ are noncentral Lie ideals of the prime ring $R$ satisfying $[K, L]=0$, it follows from Lemma \ref{lem5} that
$R$ is exceptional.

If $K$ is not abelian, by Lemma \ref{lem131} we have $\dim_CKC\geq 3$ and so
$[RC, RC]\subseteq KC$ (see Theorem \ref{thm19}). Thus
$$
\big[[RC, RC], L]=0.
$$
By Lemma \ref{lem8} (ii), $L\subseteq Z(R)$, a contradiction.
That is, $K$ is abelian and, similarly, so is $L$.

In view of Lemma \ref{lem131}, for any $a\in L\setminus Z(R)$ and $b\in K\setminus Z(R)$ we have
$$
LC=Ca+C\ \text{\rm and}\ KC=Cb+C.
$$
Hence $[a, b]=0$. Then $b\in \mathfrak{C}_{RC}(a)=Ca+C$ (see Lemma \ref{lem9}). This proves that $KC=LC=Ca+C$  for any $a\in L\setminus Z(R)$.
\end{proof}

\begin{cor}\label{cor7}
Let $R$ be an exceptional prime ring with noncentral Lie ideals $K, L$. Then $[K, L]\subseteq Z(R)$ iff $K, L\subseteq [RC, RC]$.
\end{cor}

\begin{proof}
Assume that $K, L\subseteq [RC, RC]$. In view of Lemma \ref{lem7}, $[KC, LC]\subseteq C$ and so $[K, L]\subseteq Z(R)$.

Conversely, assume that $[K, L]\subseteq Z(R)$. Let $a\in K$. If $0\ne [a, L]\subseteq Z(R)$, then, by Lemma \ref{lem11}, we see that $a\in [RC, RC]$. That is, either $a\in [RC, RC]$ or $[a, L]=0$. Hence $K$ is the union of two additive subgroups $K\cap [RC, RC]$ and $\{a\in K\mid [a, L]=0\}$.
Thus either $K\subseteq [RC, RC]$ or $[K, L]=0$. Similarly, either $L\subseteq [RC, RC]$ or $[K, L]=0$.
That is, either $K, L\subseteq [RC, RC]$ or $[K, L]=0$.

The latter case implies that $KC=LC=Cb+C$ for any $b\in L\setminus Z(R)$ (see Lemma \ref{lem133}).
By Lemma \ref{lem135}, we have $KC=LC\subseteq  [RC, RC]$, as desired.
\end{proof}

We now prove the main theorem (i.e., Theorem C) in this section.

\begin{thm}\label{thm2}
Let $R$ be a prime ring with noncentral Lie ideals $K, L$, and positive integers $m, n$. Then the following are equivalent:

(i)\ $[K^m, L^n]=0$;

(ii)\ $[K, L]=0$;

 (iii)\ $R$ is exceptional and $KC=LC=Ca+C$ for any noncentral element $a\in L$.
\end{thm}

\begin{proof}
Applying Corollary \ref{cor2}, we get that $[K^m, L^n]=0$ iff $[K, L]=0$. Therefore, (i) is equivalent to (ii). Assume that (ii) holds.
It follows from Lemma \ref{lem133}  and Lemma \ref{lem5} that $R$ is exceptional and $KC=LC=Ca+C$ for any noncentral element $a\in L$. Thus, (iii) holds.
Clearly, (iii) implies (ii) and so the proof is complete.
\end{proof}

We end this section with an application, which characterizes noncentral Lie ideals $K, L, N$ of a prime ring $R$ satisfying $\big[K^m, [L^s, N^t]\big]=0$.

\begin{thm}\label{thm9}
Let $R$ be a prime ring with noncentral Lie ideals $K, L, N$. Suppose that $\big[K^m, [L^s, N^t]\big]=0$, where $m\geq 1, s>1, t>1$ are positive integers. Then $R$ is exceptional and the following hold:

(i)\ If $\dim_CKC\geq 3$, then $\dim_CLC=2=\dim_CNC$;

(ii)\ If $\dim_CKC=2$, then $KC=LC$, $KC=NC$, or $\dim_CLC=2=\dim_CNC$.
\end{thm}

\begin{proof}
Since $\big[K^m, [L^s, N^t]\big]=0$, it follows from Lemma \ref{lem6} that
\begin{eqnarray}
\big[K, [L^s, N^t]\big]=0.
\label{eq:1}
\end{eqnarray}

\noindent{\bf Claim 1}:\ $R$ is exceptional. Otherwise, by Lemma \ref{lem5}, we have $K\subseteq Z(R)$, $L^s\subseteq Z(R)$, or $N^t\subseteq Z(R)$.
The later two cases imply that either $L\subseteq Z(R)$ or $N\subseteq Z(R)$ (see Corollary \ref{cor1}), a contradiction. Hence $R$ is exceptional.\vskip4pt

\noindent{\bf Claim 2}:\ If $[L^s, N^t]\subseteq Z(R)$ then $\dim_CLC=2=\dim_CNC$. Suppose on the contrary that either $\dim_CLC>2$ or $\dim_CNC>2$.
Without loss of generality, we assume $\dim_CLC>2$. In view of Corollary \ref{cor4}, $L^s$ contains a nonzero ideal $I$ of $R$. Thus
$[I, N^t]\subseteq Z(R)$ and so $N^t\subseteq Z(R)$. This implies that $N\subseteq Z(R)$ (see  Corollary \ref{cor1}), a contradiction.  Therefore,
$\dim_CLC=2=\dim_CNC$.

(i)\ Assume that $\dim_CKC\geq 3$. Then $RC=KC+Cb$  for some $b\in RC$. In view of Lemma \ref{lem4}, we have $[RC, RC]\subseteq KC$.
By Eq.\eqref{eq:1}, we get $\big[[RC, RC], [L^s, N^t]\big]=0$.
This implies that $[L^s, N^t]\subseteq Z(R)$ (see Lemma \ref{lem8} (ii)). By Claim 2, $\dim_CLC=2=\dim_CNC$.

(ii)\ Assume that $\dim_CKC=2$. Then it follows from Lemma \ref{lem131} that $KC=[a, RC]=Ca+C$ for some $a\in K$. Since $K$ is noncentral in $R$, one of $L^s$ and $N^t$ cannot contain any nonzero ideal of $R$.
We assume without loss of generality that $L^s$ does not contain any nonzero ideal. In view of Corollary \ref{cor4}, we have $\dim_CLC=2$ and so $LC=[b, RC]=Cb+C$ for some $b\in L\setminus Z(R)$ (see Lemma \ref{lem131}). In this case, $L^sC=LC$. It follows from Eq.\eqref{eq:1} that
\begin{eqnarray}
\big[K, [L, N^t]\big]=0.
\label{eq:2}
\end{eqnarray}
We assume that $KC\ne LC$. By Theorem \ref{thm2} and Lemma \ref{lem7}, we have
$$
0\ne [K, L]\subseteq \big[[RC, RC], [RC, RC]\big]\subseteq C,
$$
implying $0\ne [a, b]\in Z(R)$.

Suppose first that $N^t$ contains a nonzero ideal, say $I$, of $R$. It follows from Eq.\eqref{eq:2} that
\begin{eqnarray}
\big[a, [b, I]\big]=0.
\label{eq:3}
\end{eqnarray}
In view of Eq.\eqref{eq:3}, we expand $\big[a, [b, bx]\big]=0$ for all $x\in I$. Since $0\ne [a, b]\in Z(R)$ and $[a, [b, x]]=0$, we have
$[a, b][b, x]=0$ and so $[b, x]=0$. That is, $[b, I]=0$ and so $b\in Z(R)$, a contradiction.

This means that $N^t$ does not contain any nonzero ideal of $R$. By Corollary \ref{cor4}, we have $NC=Cf+C$ for some noncentral $f\in N$.
Therefore, $\dim_CLC=2=\dim_CNC$, as desired.
\end{proof}

\section{Theorem D}

We extend Theorem \ref{thm3} to the product of two noncentral Lie ideals of a prime ring.

\begin{pro}\label{pro2}
Let $R$ be a prime ring with noncentral Lie ideals $K, L$. Then $KL$ contains a nonzero ideal of $R$
except when $KC=LC=Ca+C$ for some $a\in L$.
\end{pro}

\begin{proof}
\noindent{Case 1}:\ $[K, L]$ is not abelian. Then
$$
0\ne \big[[K, L], [K. L]\big]\subseteq [K\cap L, K\cap L].
$$
In view of Theorem \ref{thm3}, $(K\cap L)^2$ contains a nonzero ideal of $R$. Since $(K\cap L)^2\subseteq KL$, $KL$
contains a nonzero ideal of $R$, as desired.\vskip6pt

\noindent{Case 2}:\ $[K, L]$ is abelian. Since $K, L$ are noncentral Lie ideals of $R$, it follows from Lemma \ref{lem5} that $R$ is exceptional.

We first consider the case that $L$ is not abelian. Let $I:=\widetilde R[L, L]\widetilde R\ne 0$.
By Lemma \ref{lem8} (vi), we have
$[I, R]\subseteq L$.

If $[K, I]\subseteq Z(R)$, then $K$ is central, a contradiction. Therefore, $[K, I]\nsubseteq Z(R)$.
Choose $a\in [K, I]\setminus Z(R)$.
Then $[a, I]\subseteq K$.
Since $[a, I]\nsubseteq Z(R)$, one can choose $b\in [a, I]\setminus Z(R)$.
Note that $I$ itself is a prime ring.
By Lemma \ref{lem8} (ii) and Lemma \ref{lem7}, there exists a nonzero
$$
\beta\in [b, [I, I]]\subseteq \big[[a, I], [I, I]\big]\subseteq Z(R)\cap K.
$$
Hence
$$
\beta\Big([I, I]+a[I, I]\Big)\subseteq \beta[I, I]+a[\beta I, I]\subseteq \beta[I, R]+a[I, R]\subseteq KL.
$$
Note that $a\in I\setminus Z(I)$. In view of Corollary \ref{cor5}, $[I, I]+a[I, I]$ contains a nonzero ideal, say $N$, of $I$.
Then $\beta INI$ is a nonzero ideal of $R$, and $\beta INI\subseteq KL$, as desired.
Similarly, we are done for the case that $K$ is not abelian.

The rest is to check the case that both $K$ and $L$ are abelian, but $KC\ne LC$.
In view of Lemma \ref{lem131}, we have
$$
KC=[a, RC]=Ca+C\ \ \text{\rm and}\ \ LC=[b, RC]=Cb+C
$$
for some $a\in K\setminus Z(R)$ and $b\in L\setminus Z(R)$. Recall that $C$ is equal to the quotient field of $Z(R)$.

We claim that there exists $0\ne \gamma\in Z(R)\cap L$. Indeed, $b\in [b, RC]$ and, by Lemma \ref{lem8} (ii), $\big[b, [R, R]\big]\ne 0$.
Therefore, there exists
$$
0\ne\gamma\in \big[b, [R, R]\big]\subseteq L\cap \big[[RC, RC], [RC, RC]\big]\subseteq C,
$$
where the last inclusion follows from Lemma \ref{lem7}. This proves our claim.
We also have
$$
[a, R]\gamma+[a, R]b\subseteq KL.
$$

Let $J:=[a, R]+[a, R]b$. Then
\begin{eqnarray}
\gamma J=\gamma [a, R]+[a,\gamma R]b\subseteq [a, R]\gamma +[a, R]b\subseteq KL.
\label{eq:5}
\end{eqnarray}
Since $KC\ne LC$, it follows from Lemma \ref{lem133} that $[K, L]\ne 0$ and so
$$
0\ne \delta:=[a, b]\in [K, L]\subseteq\big[[a, RC], [b, RC]\big]\subseteq C.
$$
Therefore, $0\ne\delta\in Z(R)\cap L$ and by Eq.\eqref{eq:5} we have
$$
 0\ne\gamma \delta R=\gamma R[a, b]\subseteq\gamma \Big([a, Rb]+[a, R]b\Big)\subseteq\gamma J\subseteq KL,
$$
 as desired.
\end{proof}

\begin{lem}\label{lem10}
Let $R$ be a prime ring with noncentral Lie ideals $K_1,\ldots,K_m$, where $m\geq 1$. Then $K_1\cdots K_m$ is also a noncentral Lie ideal of $R$.
\end{lem}

\begin{proof}
Clearly, we may assume that $m>1$. Suppose on the contrary that $K_1K_2\cdots K_m$ is a central Lie ideal of $R$. Then $K_1K_2\cdots K_m\subseteq Z(R)$ and so
$$
K_1K_2\cdots K_{m-1}[K_m, R]\subseteq K_1K_2\cdots K_{m-1}K_m\subseteq Z(R).
$$
Since $K_m\nsubseteq Z(R)$, it follows from Lemma \ref{lem8} (iii) that $K_1K_2\cdots K_{m-1}=0$. Repeating the same argument, we finally get $K_1=0$, a contradiction.
\end{proof}

We are now ready to prove the main theorem (i.e., Theorem D) in this section.

\begin{thm}\label{thm11}
Let $R$ be a prime ring with noncentral Lie ideals $K_1,\ldots,K_m$ with $m\geq 2$.
Then $K_1K_2\cdots K_m$ contains a nonzero ideal of $R$ except when, for $1\leq j\leq m$, $K_1C=K_jC=Ca+C$ for any noncentral element $a\in K_1$.
\end{thm}

\begin{proof}
We proceed the proof by induction on $m$. If $m=2$, we are done by Proposition \ref{pro2}. Assume $m>2$.
Suppose that $K_1 K_2 \cdots K_m$ does not contain any nonzero ideal of $R$.
In view of Lemma \ref{lem10}, $K_1K_2\cdots K_{m-1}$ is a noncentral Lie ideal of $R$.
It follows from Proposition \ref{pro2} that $K_1K_2\cdots K_{m-1}C = K_mC =Cb+C$ for some $b\in K_m\setminus Z(R)$. Since $\dim_CK_1K_2\cdots K_{m-1}C=2$, it follows that
$K_1K_2\cdots K_{m-1}$ does not contain any nonzero ideal of $R$. By the inductive hypothesis, we get
$$
K_1C=\cdots=K_{m-1}C=Ca+C
$$
for any $a\in K_1\setminus Z(R)$.
Since $\dim_C K_m C = 2$, we have
$K_mC=K_1K_2\cdots K_{m-1}C=Ca+C$. Hence $K_1C=K_2C=\cdots= K_mC = Ca+C$ for any $a\in K_1\setminus Z(R)$.
\end{proof}

\section{Theorem E}
Let $R$ be a prime ring with extended centroid $C$, and let $K$ be a noncentral Lie ideal of $R$.
Bergen et al. proved that $\mathfrak{C}_R(K)=Z(R)$ if $\text{\rm char}\,R\ne 2$ (see \cite[Lemma 2]{bergen1981}). See also \cite[Theorem 1]{bergen1981} and \cite[Theorem 1]{lee1983}.
Ke proved that $\mathfrak{C}_R(K)=Z(R)$ except when $R$ is exceptional (see  \cite[Lemma 2]{ke1985}).
We determine $\mathfrak{C}_R(K)$ for arbitrary prime ring $R$ as follows.

\begin{thm}\label{thm15}
Let $R$ be a prime ring with a noncentral Lie ideal $K$. Then the following hold:

(i)\ $\mathfrak{C}_{RC}(K)=KC$ iff $\dim_CKC=2$;

(ii)\ $\mathfrak{C}_{R}(K)=Z(R)$ iff $\dim_CKC>2$.
\end{thm}

\begin{proof}
(i)\ Suppose that $\dim_CKC=2$. In view of Lemma \ref{lem131}, $R$ is exceptional and $KC=Ca+C$, where $a\in K\setminus Z(R)$.
It follows from Lemma \ref{lem9} that
$$
\mathfrak{C}_{RC}(K)=\mathfrak{C}_{RC}(a)=Ca+C=KC.
$$
Conversely, assume that $\mathfrak{C}_{RC}(K)=KC$.  In particular, $K$ is abelian. It follows from Lemma \ref{lem131} that $\dim_CKC=2$, as desired.

(ii)\ Assume that $\dim_CKC>2$. In view of Corollary \ref{cor4}, $K^2$ contains a nonzero ideal of $R$. Therefore, $\mathfrak{C}_{R}(K^2)=Z(R)$.
Since $Z(R)\subseteq \mathfrak{C}_{R}(K)\subseteq \mathfrak{C}_{R}(K^2)$, we get $\mathfrak{C}_{R}(K)=Z(R)$.

Conversely, assume that $\mathfrak{C}_{R}(K)=Z(R)$. We claim that $\dim_CKC>2$. Otherwise, $\dim_CKC=2$ since $K$ is noncentral.
By (i) we have $\mathfrak{C}_{RC}(K)=KC$ and so $K\subseteq \mathfrak{C}_{R}(K)=Z(R)$, a contradiction.

Hence the proof is complete.
\end{proof}

\begin{thm}\label{thm14}
Let $R$ be a prime ring with noncentral Lie ideals $K_1,\ldots,K_m$ with $m\geq 1$.
Then either $\mathfrak{C}_R(K_1\cdots K_m)=Z(R)$ or $\mathfrak{C}_{RC}(K_1\cdots K_m)=K_1C$ where $K_1C=K_jC$ for $j=1,\ldots,m$ and $\dim_CK_1C=2$.
\end{thm}

\begin{proof}
We note that the case $m=1$ has been proved by Theorem \ref{thm15}.
Assume that $m\geq 2$.
If $K_1K_2\cdots K_m$ contains a nonzero ideal of $R$, then the primeness of $R$ implies that $\mathfrak{C}_R(K_1K_2\cdots K_m)=Z(R)$.
Otherwise, $K_1K_2\cdots K_m$ does not contain any nonzero ideal of $R$. Choose a noncentral element $a\in K_1$.  In view of Theorem \ref{thm11},
$K_1C=K_jC=Ca+C$ for $1\leq j\leq m$. In this case, $R$ is exceptional (see Lemma \ref{lem131}) and $K_1K_2\cdots K_mC=K_1C=Ca+C$.
Therefore, by Lemma \ref{lem9} we have
$$
\mathfrak{C}_{RC}(K_1K_2\cdots K_m)=\mathfrak{C}_{RC}(K_1K_2\cdots K_mC)=\mathfrak{C}_{RC}(a)=Ca+C=K_1C,
$$
as desired.
\end{proof}

Finally, the following theorem (i.e., Theorem E) is an application of both Theorem \ref{thm11} and Theorem \ref{thm14}. This also extends Theorem \ref{thm2} to its full generality.

\begin{thm}\label{thm13}
Let $R$ be a prime ring with noncentral Lie ideals $K_1,\ldots,K_m, L_1, \ldots,L_n$, where $m, n\geq 1$.
Then $\big[K_1\cdots K_m, L_1\cdots L_n\big]=0$ iff, for all $j, k$, we have $K_1C=K_jC=L_kC=Ca+C$ for any noncentral element $a\in K_1$.
\end{thm}

\begin{proof}
``$\Longrightarrow$":\ The case $m=n=1$ has been proved by Theorem \ref{thm2}. Suppose that $m>1$ or $n>1$. Without loss of generality, we assume that $n>1$.
By assumption, we have $L_1L_2\cdots L_n\subseteq \mathfrak{C}_{RC}(K_1\cdots K_m)$.
In view of Theorem \ref{thm14}, one of the following two cases holds:

\noindent{Case 1}:\ $\mathfrak{C}_R(K_1\cdots K_m)=Z(R)$. Therefore $L_1L_2\cdots L_n\subseteq Z(R)$, a contradiction (see Lemma \ref{lem10}).

\noindent{Case 2}:\ $\mathfrak{C}_{RC}(K_1\cdots K_m)=K_1C$, where $K_1C=K_jC$ for $j=1,\ldots,m$ and $\dim_CK_1C=2$.
In view of Lemma \ref{lem131}, $K_1C=Ca+C$ for any $a\in K_1\setminus Z(R)$.

In this case, we have
$$
L_1L_2\cdots L_nC\subseteq \mathfrak{C}_{RC}(K_1\cdots K_m)\subseteq K_1C.
$$
Since $\dim_CK_1C=2$ and $L_1\cdots L_nC$ is a noncentral Lie ideal of $RC$ (see Lemma \ref{lem10}), we get
$$
L_1L_2\cdots L_nC=K_1C.
$$
In particular, $L_1L_2\cdots L_n$ does not contain any nonzero ideal of $R$.
In view of Theorem \ref{thm11}, for any $b\in L_1\setminus Z(R)$, we have
$L_1C=L_kC=Cb+C$ for all $1\leq k\leq n$. By Lemma \ref{lem131}, $R$ is exceptional. Hence
$$
Ca+C=K_1C=L_1L_2\cdots L_nC=Cb+C=L_kC
$$
for all $k$, as desired.

``$\Longleftarrow$":\ Choose $a\in K_1\setminus Z(R)$. By assumption, we have $K_1C=K_jC=L_kC=Ca+C$ for all $j, k$.
Thus
$K_1\cdots K_m\subseteq \sum_{i=0}^mCa^i$ and $L_1\cdots L_n\subseteq \sum_{i=0}^nCa^i$.
Hence $\big[K_1\cdots K_m, L_1\cdots L_n\big]=0$, as desired.
\end{proof}

\end{document}